\documentclass[12pt,a4paper]{amsart}

\usepackage{hyperref} 
\usepackage{amsrefs}
\usepackage[all]{xy}

\allowdisplaybreaks

\title[Two-cycle gentle algebras]%
  {The derived equivalence classification \\ of gentle two-cycle algebras}

\author{Grzegorz Bobi\'nski}

\address{Faculty of Mathematics and Computer Science \\ Nicolaus
Copernicus University \\ ul.~Chopina 12/18 \\ 87-100 Toru\'n \\
Poland}

\email{gregbob@mat.uni.torun.pl}



\newcounter{claim}[section]

\newtheorem{corollary}[claim]{Corollary}

\newtheorem{lemma}[claim]{Lemma}
\newtheorem{proposition}[claim]{Proposition}

\newtheorem{theorem}{Theorem}

\newcommand{\bbA}{\mathbb{A}}
\newcommand{\bbN}{\mathbb{N}}
\newcommand{\bbZ}{\mathbb{Z}}

\newcommand{\bone}{\mathbf{1}}

\newcommand{\calD}{\mathcal{D}}
\newcommand{\calF}{\mathcal{F}}
\newcommand{\calO}{\mathcal{O}}
\newcommand{\calP}{\mathcal{P}}

\DeclareMathOperator{\der}{der}

\DeclareMathOperator{\rep}{rep}
\DeclareMathOperator{\gldim}{gldim}

\begin{document}

\begin{abstract}
We complete the derived equivalence classification of the gentle two-cycle algebras initiated in earlier papers by Avella-Alaminos and Bobi\'nski--Malicki.
\end{abstract}

\maketitle

\section*{Introduction and the main result}

Throughout the paper $k$ denotes a fixed algebraically closed field. By $\bbZ$, $\bbN$ and $\bbN_+$ we denote the sets of integers, nonnegative integers and positive integers, respectively. If $i$ and $j$ are integers, then $[i, j]$ denotes the set of integers $l$ such that $i \leq l \leq j$.

For a (finite-dimensional basic connected) algebra $\Lambda$ one considers its (bounded) derived category $\calD^b (\Lambda)$, which has a structure of a triangulated category. Algebras $\Lambda'$ and $\Lambda''$ are said to be derived equivalent if the categories $\calD^b (\Lambda')$ and $\calD^b (\Lambda'')$ are triangle equivalent. A study of derived categories, initiated by papers of Happel~\cites{Happel1987, Happel1988} and motivated by tilting theory, is an important direction of research in representation theory of algebras (see for example~\cites{Geiss, Happel1991, Keller, KoenigZimmermann, Ricard1989a, Ricard1989b}). It is worth to remark, that derived categories appearing in representation theory of algebras have sometimes connections with derived categories studied in algebraic geometry~\cites{Beilinson, GeigleLenzing}. One of the topics studied is derived equivalence classification of algebras (see for example~\cites{Asashiba1999, Bastian, BialkowskiHolmSkowronski, BobinskiBuan, BobinskiGeissSkowronski, BocianHolmSkowronski, Brustle, BuanVatne, Holm}).

Gentle algebras, introduced by Assem and Skowro\'nski~\cite{AssemSkowronski}, form an important class of special biserial algebras~\cite{SkowronskiWaschbusch}. For example, gentle algebras appear in a description of repre\-sentation-infinite standard biserial selfinjective algebras due to Po\-go\-rza\-\l{}y and Skowro\'nski~\cite{PogorzalySkowronski}.

Gentle algebras play also an important role in classification of algebras up to derived equivalence. Firstly, the gentle tree algebras are precisely the algebras derived equivalent to the hereditary algebras of Dynkin type $\bbA$~\cite{AssemHappel}. Secondly, the gentle one-cycle algebras which satisfy the clock condition are precisely the algebras derived equivalent to the hereditary algebras of Euclidean type $\tilde{\bbA}$~\cite{AssemSkowronski}. Finally, the gentle one-cycle algebras which do not satisfy the clock condition are precisely the discrete derived algebras, which are not locally finite~\cite{Vossieck}. The derived equivalence classes of the gentle algebras with at most one-cycle are also known and they are distinguished by the invariant of Avella-Alaminos and Geiss~\cite{AvellaAlaminosGeiss}. One should also note that the class of gentle algebras is closed with respect to the derived equivalence~\cite{SchroerZimmermann}.

Taking above into account it is natural to ask about the derived equivalence classification of the gentle two-cycle algebras. Here a gentle algebra $\Lambda$ is called two-cycle if the number of edges in the Gabriel quiver of $\Lambda$ exceeds by one the number of vertices in this quiver. Before formulating the main result we define some families of gentle two-cycle algebras.

For $p \in \bbN_+$ and $r \in [0, p - 1]$, $\Lambda_0 (p, r)$ is the algebra of the quiver
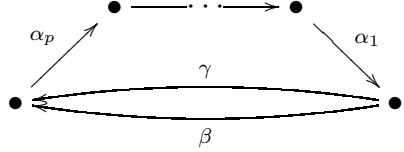
\[
\xymatrix{
& \bullet \ar[rr]|{\textstyle \cdots} & & \bullet \ar[rd]^{\alpha_1}
\\
\bullet \ar[ru]^{\alpha_p} & & & & \bullet \ar@/_/[llll]_\gamma \ar@/^/[llll]^\beta}
\]
bound by $\alpha_p \beta$, $\alpha_i \alpha_{i + 1}$ for $i \in [1, r]$ and $\gamma \alpha_1$. Moreover, for $p \geq 1$, $\Lambda_0 (p + 1, -1)$ is the algebra of the quiver
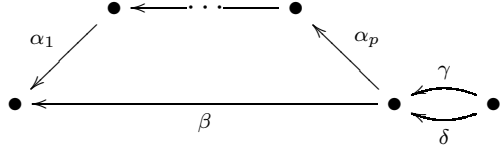
\[
\xymatrix{
& \bullet \ar[ld]_{\alpha_1} & & \bullet \ar[ll]|{\textstyle \cdots}
\\
\bullet & & & & \bullet \ar[llll]^\beta \ar[lu]_{\alpha_p} & \bullet \ar@/_/[l]_\gamma \ar@/^/[l]^\delta}
\]
bound by $\alpha_p \gamma$ and $\beta \delta$. Furthermore, for $p_1, p_2 \in \bbN_+$, $p_3, p_4 \in \bbN$, and $r_1 \in [0, p_1 - 1]$, such that $p_2 + p_3 \geq 2$ and $p_4 + r_1 \geq 1$, $\Lambda_1 (p_1, p_2, p_3, p_4, r_1)$ is the algebra of the quiver
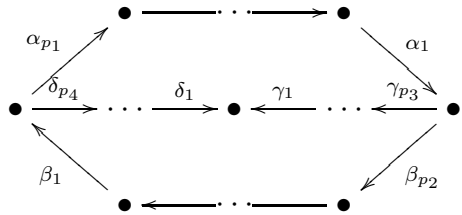
\[
\xymatrix{
& \bullet \ar[rr]|{\textstyle \cdots} & & \bullet \ar[rd]^{\alpha_1}
\\
\bullet \ar[ru]^{\alpha_{p_1}} \ar[r]^-{\delta_{p_4}} & \cdots \ar[r]^-{\delta_1} & \bullet & \cdots \ar[l]_-{\gamma_1} & \bullet \ar[l]_-{\gamma_{p_3}}
\ar[ld]^{\beta_{p_2}}
\\
& \bullet \ar[lu]^{\beta_1} & & \bullet \ar[ll]|{\textstyle \cdots}}
\]
bound by $\alpha_i \alpha_{i + 1}$ for $i \in [p_1 - r_1, p_1 - 1 ]$, $\alpha_{p_1} \beta_1$, $\beta_i \beta_{i + 1}$ for $i \in [1, p_2 - 1]$, and $\beta_{p_2} \alpha_1$. Finally, for $p_1, p_2 \in \bbN_+$, $p_3 \in \bbN$, $r_1 \in [0, p_1 - 1]$, and $r_2 \in [0, p_2 - 1]$, such that $p_3 + r_1 + r_2 \geq 1$, $\Lambda_2 (p_1, p_2, p_3, r_1, r_2)$ is the algebra of the quiver
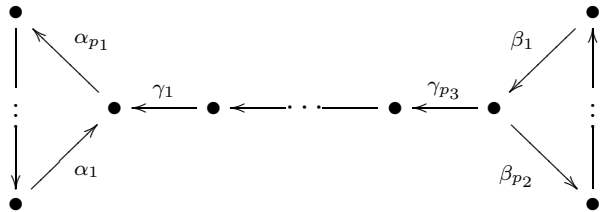
\[
\xymatrix{
\bullet \ar[dd]|{\textstyle \vdots} & & & & & & \bullet \ar[ld]_{\beta_1}
\\
& \bullet \ar[lu]_{\alpha_{p_1}} & \bullet \ar[l]_{\gamma_1} & & \bullet \ar[ll]|{\textstyle \cdots} & \bullet \ar[l]_{\gamma_{p_3}} \ar[rd]_{\beta_{p_2}}
\\
\bullet \ar[ru]_{\alpha_1} & & & & & & \bullet \ar[uu]|{\textstyle \vdots}}
\]
bound by $\alpha_i \alpha_{i + 1}$ for $i \in [p_1 - r_1, p_1 - 1]$, $\alpha_{p_1} \alpha_1$, $\beta_i \beta_{i + 1}$ for $i \in [p_2 - r_2, p_2 - 1]$, and $\beta_{p_2} \beta_1$.

The main aim of this paper is to prove the following theorem.

\begin{theorem} \label{main theorem}
The above defined algebras are representatives of the derived equivalence classes of the gentle two-cycle algebras. More precisely,
\begin{enumerate}

\item
if $\Lambda$ is a gentle two-cycle algebra, then $\Lambda$ is derived equivalent to one of the above defined algebras, and

\item
the above defined algebras are pairwise not derived equivalent.

\end{enumerate}
\end{theorem}

Parts of Theorem~\ref{main theorem} have been already proved in~\cite{BobinskiMalicki} (see also~\cite{AvellaAlaminos}). More precisely, the following claims have been proved there:
\begin{enumerate}

\item
If $\Lambda$ is a gentle two-cycle algebra, then $\Lambda$ is derived equivalent to an algebra from one of the families $\Lambda_0$, $\Lambda_1$ and $\Lambda_2$.

\item
The algebras from different families are not derived equivalent.

\item
The algebras from family $\Lambda_1$ ($\Lambda_2$) are pairwise not derived equivalent.

\end{enumerate}
Thus in order to prove Theorem~\ref{main theorem}, we have to show the following.

\begin{theorem} \label{main theorem bis}
If $p', p'' \in \bbN_+$, $r' \in [-1, p' - 1]$, $r'' \in [-1, p'' - 1]$, and $(p', r') \neq (1, -1) \neq (p'', r'')$, then the algebras $\Lambda_0 (p', r')$ and $\Lambda_0 (p'', r'')$ are not derived equivalent.
\end{theorem}

We note that one could replace derived equivalence by tilting-co\-tilting equivalence (see for example~\cite{AssemSkowronski}) in Theorems~\ref{main theorem} and~\ref{main theorem bis}. Indeed, obviously if algebras are not derived equivalent, then they are not tilting-cotilting equivalent. On the other hand, every derived equivalence obtained in~\cite{BobinskiMalicki} is realized via a tilting-cotilting.

The paper consists of two sections. In Section~\ref{section preliminaries} we recall necessary tools, including the invariant of Avella-Alaminos and Geiss, Auslander--Reiten quivers, Brenner--Butler reflections and behavior of derived equivalence under one-point coextensions. Next in Section~\ref{section proof} we prove Theorem~\ref{main theorem bis}. In the paper we use a formalism of bound quivers introduced by Gabriel~\cite{Gabriel}. For related background see for example~\cite{AssemSimsonSkowronski}.

One should remark that a partial version of Theorem~\ref{main theorem bis} has been obtained independently by Amiot~\cite{Amiot} and Kalck~\cite{Kalck}. Moreover, Amiot's result plays an important role in the proof.

\section{Preliminaries} \label{section preliminaries}

\subsection{Quivers and their representations}

By a quiver $\Delta$ we mean a set $\Delta_0$ of vertices and a set $\Delta_1$ of arrows together with two maps $s = s_\Delta, t = t_\Delta \colon \Delta_1 \to \Delta_0$ which assign to $\alpha \in \Delta_1$ the starting vertex $s \alpha$ and the terminating vertex $t \alpha$, respectively. We assume that all considered quivers $\Delta$ are locally finite, i.e.\ for each $x \in \Delta_0$ there is only a finite number of $\alpha \in \Delta_1$ such that either $s \alpha = x$ or $t \alpha = x$. A quiver $\Delta$ is called finite if $\Delta_0$ (and, consequently, also $\Delta_1$) is a finite set. For technical reasons we assume that all considered quivers $\Delta$ have no isolated vertices, i.e.\ there is no $x \in \Delta_0$ such that $s \alpha \neq x \neq t \alpha$ for each $\alpha \in \Delta_1$.

Let $\Delta$ be a quiver. If $l \in \bbN_+$, then by a path in $\Delta$ of length $l$ we mean every sequence $\sigma = \alpha_1 \cdots \alpha_l$ such that $\alpha_i \in \Delta_1$ for each $i \in [1, l]$ and $s \alpha_i = t \alpha_{i + 1}$ for each $i \in [1, l - 1]$. In the above situation we put $s \sigma := s \alpha_l$ and $t \sigma := t \alpha_1$. Moreover, we call $\alpha_1$ and $\alpha_l$ the terminating and the starting arrow of $\sigma$, respectively. Observe that each $\alpha \in \Delta$ is a path in $\Delta$ of length $1$. Moreover, for each $x \in \Delta_0$ we introduce the path $\bone_x$ in $\Delta$ of length $0$ such that $s \bone_x := x =: t \bone_x$. We denote the length of a path $\sigma$ in $\Delta$ by $\ell (\sigma)$. If $\sigma'$ and $\sigma''$ are two paths in $\Delta$ such that $s \sigma' = t \sigma''$, then we define the composition $\sigma' \sigma''$ of $\sigma'$ and $\sigma''$, which is a path in $\Delta$ of length $\ell (\sigma') + \ell (\sigma'')$, in the obvious way (in particular, $\sigma \bone_{s \sigma} = \sigma = \bone_{t \sigma} \sigma$ for each path $\sigma$). A path $\sigma_0$ is called a subpath of a path $\sigma$, if there exist paths $\sigma'$ and $\sigma''$ such that $\sigma = \sigma' \sigma_0 \sigma''$.

By a (monomial) bound quiver we mean a pair $\Lambda = (\Delta, R)$ consisting of a finite quiver $\Delta$ and a set $R$ of paths in $\Delta$, such that:
\begin{enumerate}

\item
$\ell (\rho) > 1$ for each $\rho \in R$, and

\item
there exists $n \in \bbN_+$ such that every path $\sigma$ in $\Delta$ with $\ell (\sigma) = n$ has a subpath which belongs to $R$.

\end{enumerate}
If $\Lambda = (\Delta, R)$ is a bound quiver, then by a path in $\Lambda$ we mean a path in $\Delta$ which does not have a subpath from $R$. A path $\sigma$ in $\Lambda$ is said to be maximal in $\Lambda$ if $\sigma$ is not a subpath of a longer path in $\Lambda$. The lack of isolated vertices in $\Delta$ implies that $\ell (\sigma) > 0$ for each maximal path $\sigma$ in $\Lambda$.

By a representation $V$ of a bound quiver $\Lambda = (\Delta, R)$ we mean a collection of finite-dimensional vector spaces $V_x$, $x \in \Delta_0$, and linear maps $V_\alpha \colon V_{s \alpha} \to V_{t \alpha}$, $\alpha \in \Delta_1$, such that the induced map $V_\rho \colon V_{s \rho} \to V_{t \rho}$ is zero for every $\rho \in R$. If $V$ and $W$ are representations, then a homomorphism $f \colon V \to W$ is a collection of linear maps $f_x \colon V_x \to W_x$ such that $f_{t \alpha} V_\alpha = W_\alpha f_{s \alpha}$ for every arrow $\alpha$ in $\Delta$. The category $\rep \Lambda$ of representations of $\Lambda$ is an abelian category. We call bound quivers $\Lambda'$ and $\Lambda''$ derived equivalent (and write $\Lambda' \simeq_{\der} \Lambda''$), if the derived categories $\calD^b (\rep \Lambda')$ and $\calD^b (\rep \Lambda'')$ are triangle equivalent. We will usually write shortly $\calD^b (\Lambda)$ instead of $\calD^b (\rep \Lambda)$ if $\Lambda$ is a bound quiver.

A connected bound quiver $\Lambda = (\Delta, R)$ is called gentle if the following conditions are satisfied:
\begin{enumerate}

\item
$R$ consists of paths of length $2$,

\item
for each $x \in \Delta_0$ there are at most two $\alpha \in \Delta_1$ such that $s \alpha = x$ and at most two $\alpha \in \Delta_1$ such that $t \alpha = x$,

\item
for each $\alpha \in \Delta_1$ there is at most one $\alpha' \in \Delta_1$ such that $s \alpha' = t \alpha$ and $\alpha' \alpha \not \in R$, and at most one $\alpha' \in \Delta_1$ such that $t
\alpha' = s \alpha$ and $\alpha \alpha' \not \in R$,

\item
for each $\alpha \in \Delta_1$ there is at most one $\alpha' \in \Delta_1$ such that ($s \alpha' = t \alpha$ and) $\alpha' \alpha \in R$, and at most one $\alpha' \in \Delta_1$ such ($t \alpha' = s \alpha$ and) $\alpha \alpha' \in R$.

\end{enumerate}

Let $\Lambda = (\Delta, R)$ be a gentle bound quiver. Note that a path $\alpha_1 \ldots \alpha_l$ in $\Delta$ is a path in $\Lambda$ if and only if $\alpha_i \alpha_{i + 1} \not \in R$ for all $i \in [1, l- 1]$. Taking this into account, we call a path $\alpha_1 \ldots \alpha_l$ in $\Delta$ an antipath in $\Lambda$ if $\alpha_i \alpha_{i + 1} \in R$ for all $i \in [1, l - 1]$. Again we call an anitpath $\omega$ maximal if $\omega$ is not a subpath of a longer anitpath in $\Lambda$.

\subsection{The invariant of Avella-Alaminos and Geiss} \label{subsection AAG}

Throughout this subsection $\Lambda = (\Delta, R)$ is a fixed gentle bound quiver.

By a permitted thread in $\Lambda$ we mean either a maximal path in $\Lambda$ or $\bone_x$, for $x \in \Delta_0$, such that there is at most one arrow $\alpha$ with $s \alpha = x$, there is at most one arrow $\beta$ with $t \beta = x$, and if such $\alpha$ and $\beta$ exist then $\alpha \beta \not \in R$. Similarly, by a forbidden thread we mean either a maximal antipath in $\Lambda$ or $\bone_x$, for $x \in \Delta_0$, such that there is at most one arrow $\alpha$ with $s \alpha = x$, there is at most one arrow $\beta$ with $t \beta = x$, and if such $\alpha$ and $\beta$ exist then $\alpha \beta \in R$.

Denote by $\calP$ and $\calF$ the sets of the permitted and forbidden threads in $\Lambda$, respectively. We define bijections $\Phi_1 \colon \calP \to \calF$ and $\Phi_2 \colon \calF \to \calP$. First, if $\sigma$ is a maximal path in $\Lambda$, then we put $\Phi_1 (\sigma) := \omega$, where $\omega$ is the unique forbidden thread such that $t \omega = t \sigma$ and either $\ell (\omega) = 0$ or $\ell (\omega) > 0$ and the terminating arrows of $\sigma$ and $\omega$ differ. If $\bone_x$, for $x \in \Delta_0$, is a permitted thread, there are two cases to consider. If there is an arrow $\beta$ such that $t \beta = x$ (note that such $\beta$ is uniquely determined), then $\Phi_1 (\bone_x)$ is the (unique) forbidden thread whose terminating arrow is $\beta$. Otherwise we put $\Phi_1 (\bone_x) := \bone_x$. We define $\Phi_2$ dually. Namely, if $\omega$ is a maximal anitpath, then $\Phi_2 (\omega) := \sigma$, where $\sigma$ is the permitted thread such that $s \sigma = s \omega$ and either $\ell (\sigma) = 0$ or $\ell (\sigma) > 0$ and the starting arrows of $\omega$ and $\sigma$ differ. Now, let $x \in \Delta_0$ and $\bone_x$ be a forbidden thread. If there is $\alpha \in \Delta_1$ such that $s \alpha = x$, then $\Phi_2$ is the permitted thread with starting arrow $\alpha$. Otherwise, $\Phi_2 (\bone_x) := \bone_x$. Finally, we put $\Phi := \Phi_1 \Phi_2 \colon \calF \to \calF$.

Let $\calF'$ be the set arrows in $\Delta$ which are not subpaths of any maximal antipath in $\Lambda$ (i.e.\ every antipath containing $\alpha$ can be extended to a longer antipath). For every $\alpha \in \calF'$ there exists uniquely determined $\alpha' \in \calF'$ such that $\alpha \alpha' \in R$. We put $\Phi' (\alpha) := \alpha'$. In this way we get a bijection $\Phi' \colon \calF' \to \calF'$.

Let $\calF / \Phi$ be the sets of orbits in $\calF$ with respect to the action of $\Phi$. For each $\calO \in \calF / \Phi$ we put $n_\calO := |\calO|$ and $m_\calO := \sum_{\omega \in \calO} \ell (\omega)$. Similarly, if $\calO \in \calF' / \Phi'$, then $n_\calO := 0$ and $m_\calO := |\calO|$.  Then we define $\phi_\Lambda \colon \bbN^2 \to \bbN$ by the formula:
\[
\phi_\Lambda (n, m) := | \{ \calO \in \calF / \Phi \cup \calF' / \Phi' : \text{$(n_\calO, m_\calO) = (n, m)$} \}| \qquad (n, m \in \bbN).
\]
Avella-Alaminos and Geiss has proved~\cite{AvellaAlaminosGeiss} that $\phi_\Lambda$ is a derived invariant, i.e.\ if $\Lambda'$ and $\Lambda''$ are derived equivalent gentle bound quivers, then $\phi_{\Lambda'} = \phi_{\Lambda''}$.

For a function $\phi \colon \bbN^2 \to \bbN$ we put $\| \phi \| := \sum_{(n, m) \in \bbN^2} \phi (n, m)$. If $\Lambda$ is a gentle bound quiver, then $\| \phi_\Lambda \|$ equals $|\calF / \Phi| + |\calF' / \Phi'|$. We will need the following observation.

\begin{lemma} \label{lemma gldim}
Let $\Lambda$ be a gentle bound quiver. If $\| \phi_\Lambda \| = 1$, then $\gldim \Lambda < \infty$.
\end{lemma}

\begin{proof}
If $\| \phi_\Lambda \| = 1$, then $\calF' = \emptyset$ (since $\calF \neq \emptyset$). The conditions $\calF' = \emptyset$ and $\gldim \Lambda < \infty$ are easily seen to be equivalent.
\end{proof}

\subsection{Boundary complexes}

Again let $\Lambda = (\Delta, R)$ be a gentle bound quiver. One defines the Auslander--Reiten quiver $\Gamma (\calD^b (\Lambda))$ of $\calD^b (\Lambda)$ in the following way: the vertices of $\Gamma (\calD^b (\Lambda))$ are (representatives of) the isomorphism classes of the indecomposable complexes in $\calD^b (\Lambda)$ and the number of arrows between vertices $X$ and $Y$ equals the dimension of the space of irreducible maps between $X$ and $Y$.

An indecomposable complex $X \in \calD^b (\Lambda)$ is called boundary if $X$ is perfect (i.e.\ quasi-isomorphic to a bounded complex of projective representations) and there is only one arrow in $\Gamma (\calD^b (\Lambda))$ terminating at $X$. Equivalently, one may say that in the Auslander--Reiten triangle (see~\cite{Happel1991}) terminating at $X$ the middle term is indecomposable. The invariant of Avella-Alaminos and Geiss describes the action of $\Sigma$ on the components of $\Gamma (\calD^b (\Lambda))$ containing boundary complexes. Consequently, we have the following observation.

\begin{lemma} \label{lemma boundary1}
Let $\Lambda$ be a gentle bound quiver such that $\| \phi_{\Lambda} \| = 1$. If $X$ and $Y$ are boundary complexes in $X \in \calD^b (\Lambda)$, then there exists an autoequivalence $F$ of $\calD^b (\Lambda)$ such that $F X = Y$.
\end{lemma}

\begin{proof}
If $\Lambda$ is derived equivalent to a hereditary algebry of Dynkin type $\bbA$, then the claim is well-known. Otherwise, it follows from the interpretation of $\phi_{\Lambda}$ in~\cite{AvellaAlaminosGeiss}*{Sections~5 and~6}, that there exists $p \in \bbN$ such that $\Sigma^p X$ and $Y$ belong to the same component of $\Gamma (\calD^b (\Lambda))$. This component is either a tube or of the form $\bbZ \bbA_\infty$, hence there exists $q \in \bbZ$ such that $\tau^q \Sigma^p X = Y$, where $\tau$ is the Auslander--Reiten translation. Since $\gldim \Lambda < \infty$ by Lemma~\ref{lemma gldim}, $\tau$ is an autoequivalence of $\calD^b (\Lambda)$ and the claim follows.
\end{proof}

If $\sigma$ is a path in $\Lambda$, then we have the corresponding (string) representation $M (\sigma)$ (see for example~\cite{ButlerRingel}). We have the following observation.

\begin{lemma} \label{lemma boundary2}
Let $\Lambda$ be a gentle bound quiver. If $\sigma$ is a maximal path in $\Lambda$, then $M (\sigma)$ \textup{(}viewed as a complex concentrated in degree $0$\textup{)} is a boundary complex in $\calD^b (\Lambda)$.
\end{lemma}

\begin{proof}
In the terminology of~\cite{Bobinski} (see also~\cite{BekkertMerklen}) a projective presentation of $M (\sigma)$ is given by the complex which corresponds the anitpath $\Phi_2^{-1} (\sigma)$. In particular, this implies that $M (\sigma)$ is a perfect complex in $\calD^b (\Lambda)$. Moreover, if one uses results of~\cite{Bobinski} in order to calculate the Auslander--Reiten triangle terminating at $M (\sigma)$, then one gets that its middle term is indecomposable. Alternatively, one may use the Happel functor~\cites{Happel1987, Happel1988} and well-known formulas (see for example~\cites{ButlerRingel, SkowronskiWaschbusch}) for calculating the Auslander--Reiten triangles in the stable category of the category of representations of the repetitive category $\hat{\Lambda}$ of $\Lambda$. We leave details to the reader.
\end{proof}

We formulate the following consequence.

\begin{corollary} \label{corollary boundary}
Let $\Lambda'$ and $\Lambda''$ be derived equivalent gentle bound quivers such that $\| \phi_{\Lambda'} \| = 1 = \| \phi_{\Lambda''} \|$. If $\sigma'$ and $\sigma''$ are maximal paths in $\Lambda'$ and $\Lambda''$, respectively, then there exists a derived equivalence $F \colon \calD^b (\Lambda') \to \calD^b (\Lambda'')$ such that $F (M (\sigma')) = M (\sigma'')$.
\end{corollary}

\begin{proof}
Let $G \colon \calD^b (\Lambda') \to \calD^b (\Lambda'')$ be a derived equivalence. We know from Lemma~\ref{lemma boundary2} that $M (\sigma')$ and $M (\sigma'')$ are boundary complexes in $\calD^b (\Lambda')$ and $\calD^b (\Lambda'')$, respectively. Consequently, $G (M (\sigma'))$ and $M (\sigma'')$ are boundary complexes in $\calD^b (\Lambda'')$. Thus, by Lemma~\ref{lemma boundary1}, there exists an autoequivalence $H$ of $\calD^b (\Lambda'')$ such that $H (G (M (\sigma'))) = M (\sigma'')$. We take $F = H \circ G$.
\end{proof}

\subsection{One-point coextensions}

If $\Lambda$ is a bound quiver and $M$ is a representation of $\Lambda$, then one defines a bound quiver $[M] \Lambda$, called the one-point coextension of $\Lambda$ by $M$ (see for example~\cite{BarotLenzing}). However, usually $[M] \Lambda$ is not monomial, even if $\Lambda$ is. Consequently, in the paper we only consider one-point coextensions of the form $[M (\sigma)] \Lambda$, where $\Lambda$ is a gentle bound quiver and $\sigma$ is a maximal path in $\Lambda$.

Let $\Lambda = (\Delta, R)$ be a gentle bound quiver and $\sigma$ a maximal path in $\Lambda$. We define the one-point coextension $[M (\sigma)] \Lambda$ of $\Lambda$ by $\sigma$ as follows: $[M (\sigma)] \Lambda := (\Delta', R')$, where
\begin{enumerate}

\item
$\Delta'$ is obtained from $\Delta$ by adding a new arrow $\alpha$ starting at $t \sigma$ and terminating at a new vertex $x$;

\item
$R'$ is the union of $R$ and the set of the relations $\alpha \alpha'$, where $\alpha'$ is an arrow in $\Delta$ terminating at $t \sigma$, which is not the terminating arrow of $\sigma$.

\end{enumerate}
We write shortly $[\sigma] \Lambda$ instead of $[M (\sigma)] \Lambda$ and call $[\sigma] \Lambda$ the coextension of $\Lambda$ by $\sigma$.

One easily gets the following.

\begin{lemma}
Let $\Lambda$ be gentle bound quiver. If $\sigma$ is a maximal path in $\Lambda$, then $[\sigma] \Lambda$ is a gentle bound quiver.
\end{lemma}

\begin{proof}
Exercise.
\end{proof}

We also have the following consequence of a result of Barot and Lenzing~\cite{BarotLenzing}*{Theorem~1}.

\begin{proposition} \label{proposition BL}
Let $\sigma'$ and $\sigma''$ be maximal paths in gentle bound quivers $\Lambda'$ and $\Lambda''$, respectively. If there exists a triangle equivalence $F \colon \calD^b (\Lambda') \to \calD^b (\Lambda'')$ such that $F (M (\sigma')) = M (\sigma'')$, then $[\sigma'] \Lambda'$ and $[\sigma''] \Lambda''$ are derived equivalent. \qed
\end{proposition}

\subsection{Brenner--Butler reflections}
Let $\Lambda = (\Delta, R)$ be a gentle bound quiver. Let $x$ be a vertex in $\Delta$ such that there is no $\alpha \in \Delta_1$ with $s \alpha = x = t \alpha$ and for each $\alpha \in \Delta_1$ with $s \alpha = x$ there exists $\beta_\alpha \in \Delta_1$ with $t \beta_\alpha = x$ and $\alpha \beta_\alpha \not \in R$. We define a bound quiver $\Lambda' = (\Delta', R')$ in the following way: $\Delta_0' = \Delta_0$, $\Delta_1' = \Delta_1$,
\begin{align*}
s_{\Delta'} \alpha & =
\begin{cases}
x & \text{if } t_\Delta \alpha = x,
\\
s_\Delta \beta_\alpha & \text{if } s_\Delta \alpha = x,
\\
s_\Delta \alpha & \text{otherwise},
\end{cases}
\\
t_{\Delta'} \alpha & =
\begin{cases}
s_\Delta \alpha & \text{if } t_\Delta \alpha = x,
\\
x & \text{if there exists $\beta \in \Delta_1$ such that}
\\
& \qquad \text{$t_\Delta \beta = x$, $s_\Delta \beta = t_\Delta \alpha$ and $\beta \alpha \in R$},
\\
t_\Delta \alpha & \text{otherwise},
\end{cases}
\end{align*}
and $R'$ consists of the following relations:
\begin{itemize}

\item
$\alpha \beta$, where $\alpha \beta \in R$ and $t_\Delta \alpha \neq x \neq s_\Delta \alpha$,

\item
$\alpha \beta_\alpha$, where $\alpha \in \Delta_1$ and $s_\Delta \alpha = x$,

\item
$\alpha \beta$, where $\alpha, \beta \in \Delta_1$ are such that $t_\Delta \alpha = x$ and $\gamma \beta \in R$ for some $\gamma \in \Delta_1$, $\gamma \neq \alpha$, with $t_\Delta \gamma = x$.

\end{itemize}
The following pictures, where the relations are indicated by dots, illustrate the situation: if locally (in a neighbourhood of $x$) $\Delta$ has the form
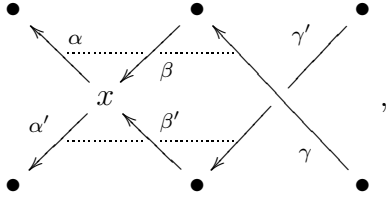
\[
\vcenter{\xymatrix@=0.5\baselineskip{
\bullet & & & & \bullet \ar[lldd]^\beta & & & & \bullet \ar[lllldddd]_(.25){\gamma'}|\hole
\\
& \ar@{.}[rr] & & \ar@{.}[rr] & &
\\
& & x \ar[lluu]_\alpha \ar[lldd]_{\alpha'}
\\
& \ar@{.}[rr] & & \ar@{.}[rr] & &
\\
\bullet & & & & \bullet  \ar[lluu]_{\beta'} & & & & \bullet \ar[lllluuuu]^(0.25)\gamma}},
\]
then locally $\Delta'$ has the form
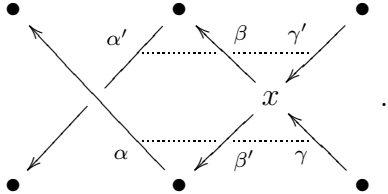
\[
\vcenter{\xymatrix@=0.5\baselineskip{
\bullet & & & & \bullet \ar[lllldddd]_(0.25){\alpha'}|\hole & & & & \bullet \ar[lldd]_{\gamma'}
\\
& & & \ar@{.}[rr] & & \ar@{.}[rr] & &
\\
& & & & & & x \ar[lluu]_\beta \ar[lldd]^{\beta'}
\\
& & & \ar@{.}[rr] & & \ar@{.}[rr] & &
\\
\bullet & & & & \bullet  \ar[lllluuuu]^(0.25)\alpha & & & & \bullet \ar[lluu]^\gamma}}.
\]
In the above situation we say that $\Lambda'$ is obtained from $\Lambda$ by applying the (generalized APR-)reflection at $x$. The bound quiver $\Lambda'$ is derived equivalent to $\Lambda$ (see~\cite{BobinskiMalicki}*{Section~1}).

We will need the following application of this operation, which is a special version of~\cite{BobinskiMalicki}*{Lemma~1.1}.

\begin{lemma} \label{lemma shift}
Let $\Lambda = (\Delta, R)$ be a gentle bound quiver such that $\Delta$ is of the form
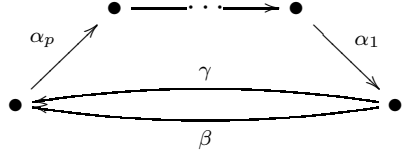
\[
\xymatrix{
& \bullet \ar[rr]|{\textstyle \cdots} & & \bullet \ar[rd]^{\alpha_1}
\\
\bullet \ar[ru]^{\alpha_p} & & & & \bullet \ar@/_/[llll]_\gamma \ar@/^/[llll]^\beta}
\]
for $p \in \bbN_+$. Assume that $\alpha_{i - 1} \alpha_i \not \in R$ and $\alpha_i \alpha_{i + 1} \in R$ for some $i \in [2, p - 1]$. Then $\Lambda$ is derived equivalent to the quiver $\Lambda' := (\Delta, R')$, where
\[
R' := (R \setminus \{ \alpha_i \alpha_{i + 1} \}) \cup \{ \alpha_{i - 1} \alpha_i \}.
\]
\end{lemma}

\begin{proof}
We apply the reflection at $t \alpha_i$.
\end{proof}

\section{Proof of the main result} \label{section proof}

The aim of this section is to prove that the bound quivers $\Lambda (p, r)$, $p \in \bbN_+$, $r \in [-1, p + 1]$, $(p, r) \neq (1, -1)$, are pairwise not derived equivalent. The following observation is crucial.

\begin{lemma} \label{lemma extension}
Let $p \in \bbN_+$ and $r \in [-1, p - 1]$, $(p, r) \neq (1, -1)$. If $\sigma$ is a maximal path in $\Lambda_0 (p, r)$, then $[\sigma] \Lambda_0 (p, r)$ is derived equivalent to $\Lambda_0 (p + 1, r)$.
\end{lemma}

\begin{proof}
Let $\sigma'$ and $\sigma''$ be maximal paths in $\Lambda_0 (p, r)$. From Lemma~\ref{lemma boundary2} we know that $M (\sigma')$ and $M (\sigma'')$ are boundary complexes in $\calD^b (\Lambda_0 (p, r))$. Note (see also~\cite{BobinskiMalicki}*{Lemma~3.1}) that $\| \phi_{\Lambda_0 (p, r)} \| = 1$. Consequently, Lemma~\ref{lemma boundary1} implies that there exists an autoequivalence $F$ of $\calD^b (\Lambda_0 (p, r))$ such that $F (M (\sigma')) = M (\sigma'')$. Now, using Proposition~\ref{proposition BL}, we get that $[\sigma'] \Lambda_0 (p, r)$ and $[\sigma'' ]\Lambda_0 (p, r)$ are derived equivalent.

It follows from the above that we may consider one particular $\sigma$. First assume that $r \geq 0$ and let $\sigma$ be the maximal path whose terminating arrow is $\beta$, i.e.\ $\sigma := \beta \alpha_1$, if $r > 0$, and $\sigma := \beta \alpha_1 \cdots \alpha_p \gamma$, if $r = 0$. Then $[\sigma] \Lambda_0 (p, r)$ is the quiver
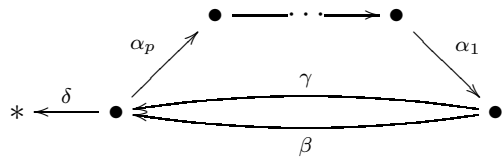
\[
\xymatrix{
& & \bullet \ar[rr]|{\textstyle \cdots} & & \bullet \ar[rd]^{\alpha_1}
\\
\ast & \bullet \ar[l]_\delta \ar[ru]^{\alpha_p} & & & & \bullet \ar@/_/[llll]_\gamma \ar@/^/[llll]^\beta}
\]
bound by relations $\alpha_p \beta$, $\alpha_i \alpha_{i + 1}$ for $i \in [1, r]$, $\gamma \alpha_1$ and $\delta \gamma$. If we apply the reflection at the vertex denoted by $\ast$, then we obtain the quiver
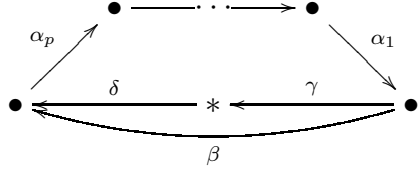
\[
\xymatrix{
& \bullet \ar[rr]|{\textstyle \cdots} & & \bullet \ar[rd]^{\alpha_1}
\\
\bullet \ar[ru]^{\alpha_p} & & \ast \ar[ll]_\delta & & \bullet \ar[ll]_\gamma \ar@/^1pc/[llll]^\beta}
\]
bound by relations $\alpha_p \beta$, $\alpha_i \alpha_{i + 1}$ for $i \in [1, r]$ and $\gamma \alpha_1$. Now we apply the reflection at the vertex denoted by $\ast$ and obtain the quiver
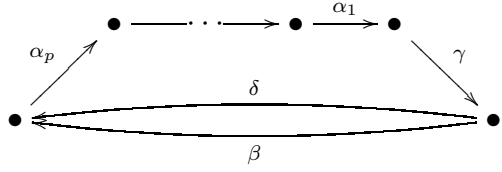
\[
\xymatrix{
& \bullet \ar[rr]|{\textstyle \cdots} & & \bullet \ar[r]^{\alpha_1} & \bullet \ar[rd]^\gamma
\\
\bullet \ar[ru]^{\alpha_p} & & & & & \bullet \ar@/_/[lllll]_\delta \ar@/^/[lllll]^\beta}
\]
bound by relations $\alpha_p \beta$, $\alpha_i \alpha_{i + 1}$ for $i \in [1, r]$ and $\delta \gamma$. Finally we shift relations (see Lemma~\ref{lemma shift}) $r$ times and obtain (the bound quiver isomorphic with) $\Lambda_0 (p + 1, r)$.

We proceed similarly if $r = -1$. Namely, we take $\beta \gamma$ as $\sigma$ and apply the refection at the terminating vertex of the new arrow. We leave details to the reader.
\end{proof}

We have the following consequence of Lemma~\ref{lemma extension}.

\begin{corollary} \label{corollary bigger}
Let $p \in \bbN_+$ and $r', r'' \in [-1, p - 1]$, $(p, r') \neq (1, -1) \neq (p, r'')$. If $\Lambda_0 (p, r')$ and $\Lambda_0 (p, r'')$ are derived equivalent, then $\Lambda_0 (q, r')$ and $\Lambda_0 (q, r'')$ are derived equivalent for all $q \geq p$.
\end{corollary}

\begin{proof}
By induction it is enough to prove that $\Lambda_0 (p + 1, r')$ and $\Lambda_0 (p + 1, r'')$ are derived equivalent, provided $\Lambda_0 (p, r')$ and $\Lambda_0 (p, r'')$ are derived equivalent. Let $\sigma'$ and $\sigma''$ be maximal paths in $\Lambda_0 (p, r')$ and $\Lambda_0 (p, r'')$, respectively. Corollary~\ref{corollary boundary} implies that there exists a derived equivalence $F \colon \calD^b (\Lambda_0 (p, r')) \to \calD^b (\Lambda_0 (p, r''))$ such that $F (M (\sigma')) = M (\sigma'')$. Thus $[\sigma'] \Lambda_0 (p, r')$ and $[\sigma''] \Lambda_0 (p, r'')$ are derived equivalent by Proposition~\ref{proposition BL}. Since $[\sigma'] \Lambda_0 (p, r') \simeq_{\der} \Lambda_0 (p + 1, r')$ and $[\sigma''] \Lambda_0 (p, r'') \simeq_{\der} \Lambda_0 (p + 1, r'')$ according to Lemma~\ref{lemma extension}, the claim follows.
\end{proof}

An important role in our proof plays the following result due to Amiot~\cite{Amiot}*{Corollary~4.4}.

\begin{proposition} \label{proposition Amiot}
Let $q \geq 3$ and $-1 \leq r', r'' \leq \frac{q}{2} - 1$. If $r' \neq r''$, then the algebras $\Lambda_0 (q, r')$ and $\Lambda_0 (q, r'')$ are not derived equivalent. \qed
\end{proposition}

Now we are ready to prove Theorem~\ref{main theorem bis}.

\begin{proof}[Proof of Theorem~\ref{main theorem bis}]
Let $p', p'' \in \bbN$, $r' \in [-1, p' - 1]$ and $r'' \in [-1, p'' - 1]$ be such that $(p', r') \neq (1, - 1) \neq (p'', r'')$. Obviously, $\Lambda_0 (p', r')$ and $\Lambda_0 (p'', r'')$ are not derived equivalent if $p' \neq p''$ (e.g.\ they have different numbers of vertices). Thus assume that $p' = p''$ and denote this common value by $p$. Choose $q \geq p$ such that $r', r'' \leq \frac{q}{2} - 1$. If $\Lambda_0 (p, r')$ and $\Lambda_0 (p, r'')$ are derived equivalent, then Corollary~\ref{corollary bigger} implies that $\Lambda_0 (q, r')$ and $\Lambda_0 (q, r'')$ are derived equivalent as well. Consequently, $r' = r''$ according to Proposition~\ref{proposition Amiot} and the claim follows.
\end{proof}

\bibsection

\end{document}